\documentclass[12pt,reqno]{amsart}

\usepackage{amssymb}
\usepackage{amsmath,amscd}
\usepackage{color}
\usepackage{epsf}
\usepackage{graphicx}

\theoremstyle{plain}
\newtheorem{thm}{Theorem}[section]
\newtheorem{lem}[thm]{Lemma}

\newtheorem{cor}[thm]{Corollary}

\theoremstyle{definition}
\newtheorem{definition}[thm]{Definition}

\newtheorem{rem}[thm]{Remark}

\newcommand{\Z}{\mathbb Z}

\begin{document}

\title [\ ] {A property that characterizes Euler characteristic among Invariants
    of combinatorial manifolds}

\author{Li Yu}
\address{Department of Mathematics and IMS, Nanjing University, Nanjing, 210093, P.R.China}
\email{yuli@nju.edu.cn}

\thanks{This work is partially supported by a grant from NSFC
(No.10826040)}

\date{March 10, 2009}



\begin{abstract}
   If a real value invariant of compact combinatorial
   manifolds (with or without boundary) depends only on the number of
   simplices in each dimension on the manifold,
   then the invariant is completely determined by Euler characteristics of
  the manifold and its boundary. So essentially, Euler characteristic
  is the unique invariant of this type.
\end{abstract}

\maketitle

  \section{Introduction}

  \begin{thm} [Roberts 2002, ~\cite{Justin02}] \label{thm:Justin}
   Any linear combination
    of the numbers of simplices which is an invariant of closed combinatorial manifolds
    must be proportional to the Euler characteristic.
    \end{thm}

     Then it is natural to ask the following:
     \vskip .2cm

    \textbf{Question 1:}  Can we find any new invariants
    of comibnatorial manifolds independent from Euler characteristic
    among nonlinear functions on the numbers of
    simplices of the manifold?\vskip .2cm

   \begin{thm} [Pachner 1986, ~\cite{Pa86}] \label{thm:Pachner}
     Two closed combinatorial n-manifolds are PL-
homeomorphic if and only if it is possible to move between their
triangulations using a sequence of bistellar moves (Pachner moves)
and simplicial isomorphisms.
  \end{thm}

   \begin{thm} [Pachner 1991, ~\cite{Pa91}] \label{thm:Pachner_2}
   Two connected combinatorial
    $n$-manifolds with non-empty boundary are PL-homeomorphic
    if and only if they are related by a sequence of
   elementary shellings, inverse shellings and a simplicial
   isomorphism.
  \end{thm}

   Here, $PL$ is an abbreviation for \textit{piecewise linear}.
   Pachner's theorems suggest that we can look for quantities invariant
   under the bistellar moves and elementary shellings
   so as to obtain invariants of combinatorial manifolds under $PL$ isomorphisms.
   However, we will give a negative answer to Question 1 by showing
   the following.
 \vskip .2cm

 \begin{thm} \label{thm:main}
        Suppose $\Phi$ is a real value invariant of closed combinatorial manifolds
        (under PL isomorphisms) that only depends on the number of simplices in each dimension on the
        manifold. If two closed combinatorial
        $n$-manifolds $(M^n_1,\xi_1)$ and $(M_2^n, \xi_2)$ have the same
        Euler characteristic, $\Phi(M^n_1,\xi_1) = \Phi(M^n_2,\xi_2)$.
 \end{thm} \vskip .2cm

   So essentially there is a unique $PL$-invariant of closed combinatorial
  manifolds whose value depends only on the number of simplices in
  the manifold ----- Euler characteristic.
  And since we have combinatorial manifolds with the same dimension and
  Euler characteristic but
   different rational Pontryagin numbers, so there could not be any formulae for
   rational Pontryagin numbers that depend only on the number of simplices
   in a combinatorial manifold. Furthermore,
   Theorem~\ref{thm:main} can be extended to
   combinatorial manifolds with nonempty boundaries as following.
  \vskip .2cm

  \begin{thm} \label{thm:main2}
        Suppose $\Phi$ is a real value invariant of compact
        combinatorial manifolds (under PL isomorphisms)
        that only depends on the number of simplices in each dimension on the
        manifold. If two compact combinatorial
        $n$-manifolds $(M^n_1,\xi_1)$ and $(M_2^n, \xi_2)$ with nonempty boundary
        have the same Euler characteristic
        $\chi(M^n_1) = \chi(M^n_2)$ and $\chi(\partial M^n_1) = \chi (\partial M^n_2)$,
        then $\Phi(M^n_1,\xi_1) = \Phi(M^n_2,\xi_2)$.
 \end{thm} \vskip .2cm

  \begin{rem}
    A real value invariant $\Phi$ of closed combinatorial manifolds
      as in Theorem~\ref{thm:main} may not admit any local formula
      (see~\cite{YuLi_Euler_1}).
        So the results in this paper and the results in~\cite{YuLi_Euler_1} are not overlapping.
  \end{rem}

  The paper is organized as follows. In Section~\ref{Sec2}, we
  briefly discuss some basic properties of $\mathbf{f}$-vectors of combinatorial
  manifolds and understand the change of
  the $\mathbf{f}$-vectors under bistellar moves and elementary shellings.
   Then in Section~\ref{Sec3} and Section~\ref{Sec4}, we present a proof of Theorem~\ref{thm:main}
  and Theorem~\ref{thm:main2} respectively.
 \\

  \section{Dehn-Sommerville Equations, Bistellar move and Elementary shelling}
   \label{Sec2}

   We first recall some definitions in combinatorial topology (see ~\cite{Brond83}
   and~\cite{RourSand74}).\vskip
   .2cm

     For a topological $n$-manifold $M^n$,
       let $\Xi(M^n)$ be the set of
        all isomorphism classes of combinatorial structures on $M^n$.
        For any $\xi \in \Xi(M^n)$, $(M^n,\xi)$ is called a
        combinatorial manifold.
        In this paper, we only consider those manifolds which admit combinatorial
        manifold structures. If $M^n$ has nonempty boundary, then we
        use $(\partial M^n, \partial \xi)$ to denote the restriction of $\xi$
        on $\partial M^n$.\vskip .2cm

        Let $f_k(M^n,\xi)$ be the number of $k$-simplices in $\xi$. We
        call the following integral vector $\mathbf{f}$\textbf{-vector} of the combinatorial
        manifold $(M^n,\xi)$.
        \[  \mathbf{f}(M^n, \xi) = (f_0(M^n,\xi), \cdots, f_n(M^n,\xi) ) \in \Z^{n+1}_{+} \]
          In addition, we define $f_{-1}(M^n, \xi) := \frac{1}{2}
         \chi(M^n)$, where $\chi(M^n)$ is the Euler
         characteristic of $M^n$.\vskip .2cm

         Following are some well-known facts on the $\mathbf{f}$-vectors of
  combinatorial manifolds (see~\cite{ChenYan97} and ~\cite{Brond83}).

     \begin{thm} [Dehn-Sommerville Equations]
      Suppose $(M^n, \xi)$ is an $n$-dimensional closed combinatorial manifold. Then we have
    \begin{equation} \label{Equ:DehnSome_1}
      f_i(M^n, \xi) = \sum^n_{j=i} (-1)^{n-j} \binom{j+1}{i+1} \cdot f_j(M^n,
      \xi),\quad -1\leq i\leq n
    \end{equation}
    If $(M^n, \xi)$ is an $n$-dimensional compact combinatorial manifold
    with boundary, we have:
   \begin{equation} \label{Equ:DehnSome_2}
      f_i(M^n, \xi) - f_i(\partial M^n, \partial\xi) =
       \sum^n_{j=i} (-1)^{n-j} \binom{j+1}{i+1} \cdot f_j(M^n, \xi)
    \end{equation}
     \end{thm}

      Note that setting $k =-1$, Dehn-Sommerville
   Equation~\eqref{Equ:DehnSome_1} yields the Euler formula for
   $M^n$. So when two closed combinatorial manifolds have the same Euler characteristic,
   their $\mathbf{f}$-vectors are solutions of exactly the same system of
   Dehn-Sommerville equations.\vskip .2cm

     \begin{cor} \label{Cor:depend}
      Suppose $(M^n, \xi)$ is an $n$-dimensional closed combinatorial
      manifold, then $f_{[\frac{n+1}{2}]}(M^n, \xi), \cdots , f_{n}(M^n,\xi)$ are completely
      decided by $f_{-1}(M^n,\xi), f_0(M^n,\xi), \cdots, f_{[\frac{n-1}{2}]}(M^n,\xi)$.
      Indeed, there are some universal constant $q_{ij} \in \Z $
      such that
      \[ f_i(M^n,\xi) = \sum^{[\frac{n-1}{2}]}_{j=-1} q_{ij} f_j(M^n,\xi),\quad  \,
         \left[ \frac{n+1}{2} \right] \leq i \leq n \]
     \end{cor}
     \begin{proof}
        The proof is parallel to that for the boundary of
         simplicial polytopes
        (see~\cite{Brond83} chapter 17 for details).
     \end{proof} \vskip .2cm

      \begin{cor} \label{Cor:depend_2}
      Suppose $(M^n, \xi)$ is an $n$-dimensional combinatorial
      manifold with boundary, then each
      $$ f_i(M^n, \xi)- \frac{1}{2} f_i(\partial M^n, \partial\xi),\
         \left[\frac{n+1}{2}\right]  \leq i \leq n $$
        is completely
      decided by the $[\frac{n+1}{2}]$ numbers:
      \[  f_j(M^n, \xi)- \frac{1}{2} f_j(\partial M^n, \partial\xi), \quad  -1\leq j \leq
      \left[\frac{n-1}{2}\right]. \]
     \end{cor}
     \begin{proof}
          Notice that $f_i(M^n, \xi)- \frac{1}{2} f_i(\partial M^n,
          \partial\xi) = \frac{1}{2} f_i (DM^n, D\xi)$ where $(DM^n, D\xi)$ is the
          gluing of two copy of $(M^n,\xi)$ along their boundaries.
          Then the result here follows from Corollary~\ref{Cor:depend} above.
     \end{proof} \vskip .2cm

     \begin{definition}
        Let $(M^n,\xi)$ be a combinatorial $n$-manifold and $\sigma \in
        \xi$ an $(n-i)$-simplex ($0\leq i \leq n$) such that its link
        in $\xi$ is the boundary $\partial\tau$ of an
        $i$-simplex $\tau$ that is not a face of $\xi$. Then the
        operation
        \[   T^{n,i}_{\sigma,\tau} (\xi) := (\xi \backslash (\sigma * \partial\tau)) \cup (\partial\sigma
        * \tau)  \]
         is called an $n$-dimensional \textbf{\textit{bistellar} $i$-\textit{move}}
         (or \textbf{\textit{Pachner}} move).
         Bistellar $i$-moves with $i\geq [\frac{n}{2}]$ is also
         called \textit{reverse bistellar} $(n-i)$-\textit{move}.
         See Figure~\eqref{p:2dim_Move} and Figure~\eqref{p:3dim_Move}.
  \end{definition}
   \vskip .2cm

      \begin{figure}
         \includegraphics[width=0.90\textwidth]{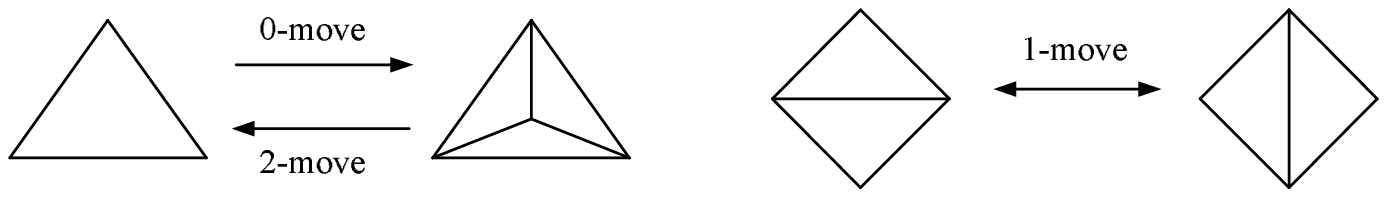}\\
          \caption{Bistellar moves in dimension $2$}\label{p:2dim_Move}
      \end{figure}

       \begin{figure}
         \includegraphics[width=\textwidth]{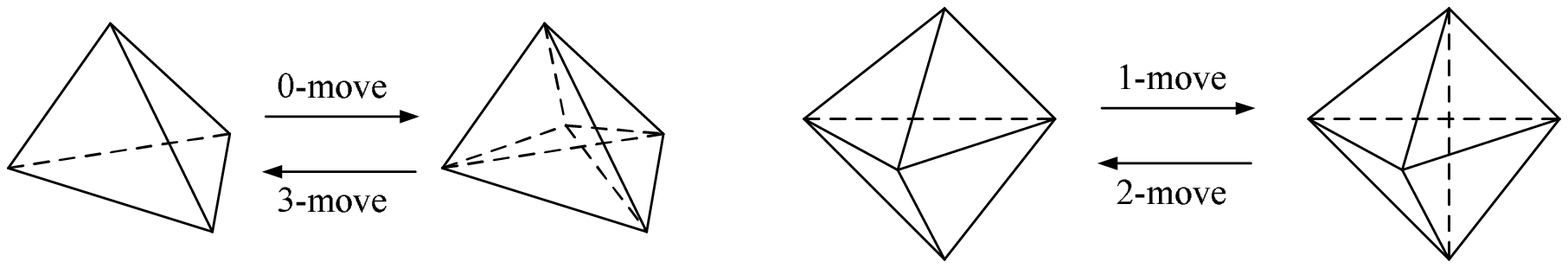}\\
          \caption{Bistellar moves in dimension $3$}\label{p:3dim_Move}
      \end{figure}

       \begin{definition}
         Suppose that $\sigma$ and $\tau$ are simplexes of a combinatorial
         $n$-manifold $M^n$ with boundary which satisfy:
         \begin{enumerate}
           \item the join $\sigma * \tau$ is an $n$-simplex of $M^n$;
           \item $\tau \cap \partial M^n = \partial\tau$ and
                $\sigma * \partial\tau \subset \partial M^n$.
         \end{enumerate}
            Then the manifold
     obtained from $M^n$ by the \textbf{\textit{elementary shelling}} $S^{n,i}_{\sigma,\tau}$
     from $\sigma$ is the closure of $M \backslash (\sigma * \tau)$, denoted by
     $S^{n,i}_{\sigma,\tau}(M^n)$. And we
     call $S^{n,i}_{\sigma,\tau}$ an $n$-dimensional elementary $i$-shelling where
     $i = \mathrm{dim}(\sigma)$ (note that $0\leq i \leq n-1$).
       \end{definition}

        \begin{figure}
         \includegraphics[width=0.62\textwidth]{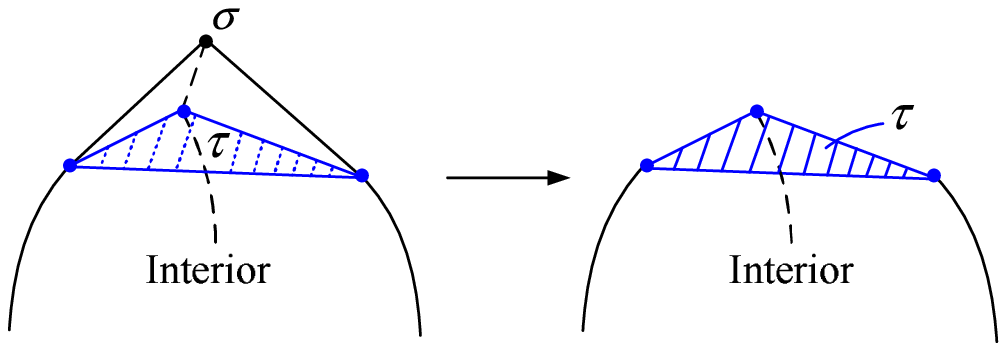}\\
          \caption{ Elementary $0$-shelling in dimension $3$ }\label{p:Shelling_0}
      \end{figure}

        \begin{figure}
         \includegraphics[width=0.62\textwidth]{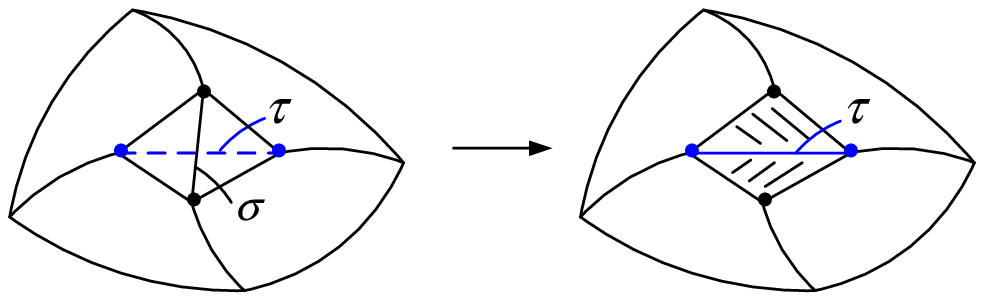}\\
          \caption{ Elementary $1$-shelling in dimension $3$ }\label{p:Shelling_1}
      \end{figure}

        \begin{figure}
         \includegraphics[width=0.64\textwidth]{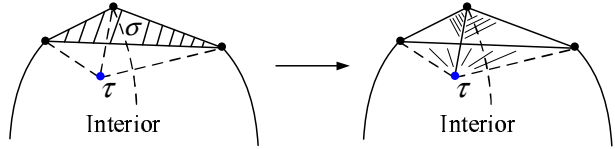}\\
          \caption{ Elementary $2$-shelling in dimension $3$ }\label{p:Shelling_2}
      \end{figure}

      Notice that on the boundary, $\partial S^{n,i}_{\sigma,\tau}(M^n)$ and
      $\partial M^n$ are related by an $(n-1)$-dimensional bistellar
      $(n-1-i)$-move (See Figure~\ref{p:Shelling_0}, Figure~\ref{p:Shelling_1}
      and Figure~\ref{p:Shelling_2} for the $3$-dimensional pictures). \\

       Suppose $T^{n,i}$ is a bistellar $i$-move on $(M^n,\xi)$, let
       $T^{n,i}(\xi)$ be the combinatorial structure on $M^n$ after the
       move. Then it is easy to see
        \[   f_k(M^n,T^{n,i}(\xi)) = f_k(M^n, \xi) + d_{k,i},\]
        where $d_{k,i} =  \binom{n+1-i}{k-i} - \binom{i+1}{n+1-k}$.
       It is easy to see that:
       \begin{equation} \label{Equ:d1}
          d_{k, n-i} = - d_{k,i},\ \forall \, 0 \leq i , k\leq n
      \end{equation}
       \begin{equation} \label{Equ:d2}
       \text{If}\ 2i= n, d_{k,i} =0, \ \forall \, 0 \leq  k\leq n
        \end{equation}

   \begin{definition}
      Let $\mathbf{d}_i  = (d_{0,i} ,\cdots , d_{n,i}) \in \Z_+^{n+1}$.
        For any element $\mathbf{f}\in \Z_+^{n+1}$, the operation
       $\beta^i(\mathbf{f}) = \mathbf{f} + \mathbf{d}_i$ is called a
       \textbf{\textit{virtual $i$-move}} in $\Z_+^{n+1}$.
   \end{definition}

    \begin{lem} \label{Lem:Virtual_Move}
         For two closed combinatorial $n$-manifolds $(M^n_1,\xi_1)$
         and $(M^n_2,\xi_2)$, if $\chi(M^n_1) = \chi (M^n_2)$, then there
         exists a finite sequence of virtual $i$-moves in
         $\Z_+^{n+1}$ ($0\leq i \leq n$)
         which turn $\mathbf{f}(M^n_1,\xi_1)$ to
         $\mathbf{f}(M^n_2,\xi_2)$.
   \end{lem}
   \begin{proof}
     It amounts to find $(m_0,\cdots, m_{n}) \in \Z^{n+1}$
        such that
       \begin{equation} \label{Equ:steps}
          \sum^{n}_{i=0} m_i \cdot d_{k,i} = f_k(M^n_2, \xi_2) -
          f_k(M^n_1,\xi_1),\quad \forall\, 0 \leq k \leq n
     \end{equation}
    By Equation~\eqref{Equ:d1}, $d_{k, n-i} = - d_{k,i}$, so we have
     \begin{equation} \label{Equ:System1}
          \sum^{[\frac{n-1}{2}]}_{i=0} (m_i - m_{n-i}) \cdot d_{k,i} = f_k(M^n_2, \xi_2) -
          f_k(M^n_1,\xi_1),\quad \forall\, 0 \leq k \leq n
     \end{equation}

     When $\chi(M^n_1) = \chi(M^n_2)$, by Corollary~\ref{Cor:depend}, the $\mathbf{f}$-vector
     $\mathbf{f}(M^n_1,\xi)$ is decided by $f_0(M^n_1,\xi),\cdots,
           f_{[\frac{n-1}{2}]}(M^n_1,\xi)$ exactly the same way as
     $\mathbf{f}(M^n_2,\xi)$ is decided by $f_0(M^n_2,\xi),\cdots,
           f_{[\frac{n-1}{2}]}(M^n_2,\xi)$.
           In addition, since the constrain defined by
          each Dehn-Sommerville equation on $\Z^{n+1}_{+}$
           is invariant under any virtual moves (see~\cite{Pa86}),
            the solution to above system~\eqref{Equ:System1}
           of $n+1$ linear equations is the same as the solution to the
            $[\frac{n-1}{2}] + 1=[\frac{n+1}{2}]$ linear equations:
             \begin{equation} \label{Equ:System2}
                \overset{[\frac{n-1}{2}]}{\underset{i=0}{\sum}}
                (m_i - m_{n-i} ) \cdot d_{k,i} = f_k(M^n_2, \xi_2) -
          f_k(M^n_1,\xi_1),\quad \forall\, 0 \leq k \leq \left[\frac{n-1}{2} \right].
            \end{equation}

    Notice when $0\leq i\leq [\frac{n-1}{2}]$, $0\leq k\leq
       [\frac{n-1}{2}]$, $d_{k,i}= \binom{n+1-i}{k-i}$. So
        \begin{itemize}
         \item if $k< i$, $d_{k,i} =0$.
         \item if $k=i$, $d_{i,i} = 1$.
        \end{itemize}
          The square integral matrix $(d_{k,i})_{0\leq k,i\leq [\frac{n-1}{2}]}$
       is invertible over $\Z$. So the above system of $[\frac{n+1}{2}]$ linear
       equations~\eqref{Equ:System2} has a unique solution, denoted
       by $x_0,\cdots, x_{[\frac{n-1}{2}]}$. Then
       $(x_0,\cdots, x_{[\frac{n-1}{2}]}, 0, \cdots, 0)$ is a solution
       of Equation~\eqref{Equ:steps}.
       Obviously, applying $|x_i|$ times virtual
       $i$-moves (or $(n-i)$-moves if $x_i < 0$)
       for all $0\leq i \leq n$ to $\mathbf{f}(M^n_1,\xi_1)$
       in $\Z^{n+1}_+$ will turn it to $\mathbf{f}(M^n_2,\xi_2)$.
   \end{proof} \vskip .2cm

    \begin{rem}
       The solution to Equation~\eqref{Equ:steps} is not unique since
       adding the same integer to $m_i$ and $m_{n-i}$ simultaneously
         in a solution $(m_0,\cdots, m_{n})$ of Equation~\eqref{Equ:steps} will give a new
         solution. But here we only need the existence of the solution.
     \end{rem}

\section{Proof of Theorem~\ref{thm:main}} \label{Sec3}

   First, we introduce a special kind of $PL$ disks in each dimension
     which will be used as an auxiliary tool in our argument later.\vskip .2cm

     \begin{lem} \label{Lem:PlumpCell}
      For any $n\geq 1$, there exists a $PL$ $n$-disk $K^n$ such that:
        \begin{enumerate}
            \item $\partial K^n$ is isomorphic to the boundary of an
                   $n$-simplex.
            \item for any $0\leq i \leq n$, there exists a bistellar
            $i$-move $T^{n,i}$ in $K^n$ which does not cause any changes to the star of any
            vertex on $\partial K^n$.
        \end{enumerate}
    \end{lem}
    \begin{proof}
        For each $0\leq i \leq n$, let $\Delta^i$ be a simplex of
        dimension $i$ and let $J^n_i = \Delta^{n-i} * \partial\Delta^i $.
        Suppose $\widetilde{\Delta}^n_1, \widetilde{\Delta}^n_2$ are
         two $n$-simplices such that $\widetilde{\Delta}^n_2$ is contained
         in the interior of $\widetilde{\Delta}^n_1$.
         Next, we put each $J^n_i$ inside $\widetilde{\Delta}^n_2$
          such that they do not intersect. By making up some new simplices and
          adding some new vertices inside $\Delta^n_2$ if necessary,
          we get a PL $n$-disk $K^n$ (see Figure~\ref{p:Magic_Cell} for a
          construction of $K^2$). The subdivision in $\widetilde{\Delta}^n_1 -
          (\widetilde{\Delta}^n_2)^{\circ}$ is canonical such that
          the triangulations on $\partial \widetilde{\Delta}^n_1$
          and $ \partial \widetilde{\Delta}^n_2$ are not changed. So
          $\partial K^n$ is isomorphic to the boundary of an
          $n$-simplex.

          Obviously, there is a bistellar
          $i$-move $T^{n,i}$ (replaceing $J^n_i$ by $\partial \Delta^{n-i} * \Delta^i$)
          associated to each $J^n_i$ inside $\Delta^n_2$. And
          $T^{n,i}$ does not change the the star of any
            vertex on $\partial K^n = \partial \widetilde{\Delta}^n_1 $.
    \end{proof} \vskip .2cm

     \begin{figure}
         \includegraphics[width=0.73\textwidth]{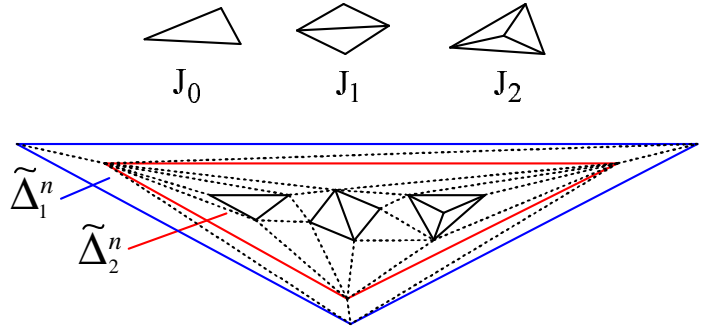}\\
          \caption{A construction of $K^2$}\label{p:Magic_Cell}
      \end{figure}

    Note that our construction of $K^n$ is far from unique. Here, we only
        need to choose one such $K^n$ in each dimension $n$. We call $K^n$
         the \textbf{\textit{plump $n$-cell}}. \vskip .2cm

  \begin{lem} \label{Lem:Change_f-vector}
         For two closed combinatorial $n$-manifolds $(M^n_1,\xi_1)$
         and $(M^n_2,\xi_2)$,
         if $\chi(M^n_1) = \chi (M^n_2)$, then there exists a
         combinatorial manifold structure $\xi^*_i$ on $M^n_i$ ($i=1,2$)  such that
           $(M^n_i, \xi^*_i) $ is PL-homeomorphic to $(M^n_i, \xi_i)$ and
             $\mathbf{f}(M^n_1, \xi^*_1)=\mathbf{f}(M^n_2, \xi^*_2)$.
         \end{lem}
   \begin{proof}
          First of all, by the Lemma~\ref{Lem:Virtual_Move}, there exist finite
          number of virtual moves in $\Z^{n+1}_+$ which turn $\mathbf{f}(M^n_1, \xi_1)$
           to $\mathbf{f}(M^n_2, \xi_2)$. Suppose the total number of virtual moves
           used is $N$.
           But it is not clear
           whether we can find $N$ bistellar moves on $(M^n_1,\xi_1)$ that
           realize these virtual moves on its $\mathbf{f}$-vector.

           To solve the problem, we can first apply the same number of bistellar $0$-moves
           on $\xi_1$ and $\xi_2$ so that the new triangulations $\xi'_1,
           \xi'_2$ each contains more than $N$ $n$-simplices. Then we choose arbitrary
           $N$ $n$-simplices in $\xi'_1$ and $\xi'_2$ respectively and subdivide
           each of the $N$ $n$-simplices into the plump cell $K^n$ (see Lemma~\ref{Lem:PlumpCell}).
           Denote the resulting combinatorial structure on $M^n_i$ by $\xi_i''$ ($i=1,2$).
           Notice
           \[ \mathbf{f}(M^n_2, \xi''_2)- \mathbf{f}(M^n_1, \xi_1'')  =
            \mathbf{f}(M^n_2, \xi_2)- \mathbf{f}(M^n_1, \xi_1).
              \]
            So we need exactly the same number and type of virtual moves
            in $\Z^{n+1}_{+}$ to turn $\mathbf{f}(M^n_1, \xi_1'')$
           to $\mathbf{f}(M^n_2, \xi_2'')$ as we turn $\mathbf{f}(M^n_1, \xi_1)$
           to $\mathbf{f}(M^n_2, \xi_2)$.
           Now, each of the $N$ plump cells in $M^n_i$ contains a bistellar move of all types,
            so we can realize those $N$ virtual moves on $\mathbf{f}(M^n_1,
           \xi_1'')$ in $\Z_+^{n+1}$ by $N$ bistellar moves on the triangulation.
            In the end, we get a combinatorial structure $\xi^*_1$ on $M^n_1$
           with $\mathbf{f}(M^n_1, \xi^*_1) = \mathbf{f}(M^n_2, \xi''_2)$.
           Set $\xi^*_2 = \xi''_2$, obviously $(M^n_i, \xi^*_i) $ is
           PL-homeomorphic to $(M^n_i, \xi_i)$, $i=1,2$. So we
           are done.
     \end{proof}
   \vskip .2cm

  \textbf{Proof of Theorem~\ref{thm:main}:}
   By the assumption, $\chi(M^n_1) = \chi(M^n_2)$, so by Lemma~\ref{Lem:Change_f-vector},
   we can assume $\mathbf{f}(M^n_1, \xi_1) = \mathbf{f}(M^n_2,\xi_2)$.
   Since the invariant $\Phi$ depends only on the $\mathbf{f}$-vector of a combinatorial manifold,
   then $\Phi(M^n_1,\xi_1) = \Phi(M^n_2,\xi_2)$. \hfill $\square$

    \ \\

  \section{Proof of Theorem~\ref{thm:main2}}  \label{Sec4}

    We first construct a special $PL$ $n$-disk associated to the plump $(n-1)$-cell
    $K^{n-1}$ in each dimension $n\geq 2$. It will be useful later
    in our argument.

    \begin{lem} \label{Lem:PlumpCell}
      For any $n\geq 2$, there exists a $PL$ $n$-disk $U^n$ such that:
        \begin{enumerate}
            \item the boundary of $U^n$ is isomorphic to a $PL$ $(n-1)$-sphere
            that is got from subdividing an $(n-1)$-face on the boundary of an $n$-simplex
            into a plump cell $K^{n-1}$. We denote the $PL$ $(n-1)$-disk
            on $\partial U^n$ which is isomorphic to a $K^{n-1}$ by $Z^{n-1}$. \vskip .1cm

            \item for any $0\leq i \leq n-1$, there exists an $n$-dimensional
             elementary $(n-1-i)$-shelling $S^{n,n-1-i}$ on $U^n$ which induces
             the bistellar $i$-move $T^{n-1, i}$ in $ Z^{n-1} \cong K^{n-1}$, and
             $S^{n,n-1-i}$ does not cause any changes to the star of any
            vertex in $\partial U^n - (Z^{n-1})^{\circ}$.
        \end{enumerate}
    \end{lem}
    \begin{proof}
       Let $\Delta^i$ be a simplex of
        dimension $i$ and $E^n_i = \Delta^{n-1-i} * \Delta^i $ is an $n$-simplex.
        Let $J^{n-1}_i = \Delta^{n-1-i} *\partial \Delta^i \subset \partial E^n_i
        $ for each $0\leq i \leq n-1$.
        On the other hand, suppose $\widetilde{\Delta}^n_1, \widetilde{\Delta}^n_2$ are
         two $n$-simplices and $\widetilde{\Delta}^{n-1}_1$, $\widetilde{\Delta}^{n-1}_2$
         is an $(n-1)$-face of $\widetilde{\Delta}^n_1, \widetilde{\Delta}^n_2$
         respectively.
          We can put $\widetilde{\Delta}^n_2 $ in $\widetilde{\Delta}^n_1$
        such that $\widetilde{\Delta}^n_2 \cap \partial\widetilde{\Delta}^n_1 =
        \widetilde{\Delta}^{n-1}_2 \subset \widetilde{\Delta}^{n-1}_1$.
         Next, we put each $E^n_i$ in $\widetilde{\Delta}^n_2$
         such that $E^n_i \cap \partial\widetilde{\Delta}^n_2 = J^{n-1}_i \subset
          \widetilde{\Delta}^{n-1}_2$ and $E^n_i \cap E^n_{i'} = \varnothing $
          (see Figure~\ref{p:Un}).

          Now, by making up some new faces and vertices
           in $\widetilde{\Delta}^n_1- \cup_i E^n_i$ if necessary,
           we can subdivide $\widetilde{\Delta}^n_1$
           into a $PL$ $n$-disk $U^n$ such that: on the $(n-1)$-face
              $\widetilde{\Delta}^{n-1}_1$, the subdivision turns
              it into a $PL$ $(n-1)$-disk $Z^{n-1}$ that is isomorphic to
              the plump $(n-1)$-cell $K^{n-1}$ (see
              Lemma~\ref{Lem:PlumpCell}),
              and all other $(n-1)$-faces of $\partial \widetilde{\Delta}^n_1 $
              are not subdivided. \vskip .1cm

          Obviously, the elementary $(n-1-i)$-shelling $S^{n,n-1-i}$ associated to
          $E^n_i$ induces the $(n-1)$-dimensional bistellar $i$-move $T^{n-1,i}$ on
          $Z^{n-1}\cong K^{n-1}$. And by our construction, the elementary shelling
          $S^{n,n-1-i}$ only changes
          the stars of the vertices on $\partial U^n$ that lie in
          $(Z^{n-1})^{\circ}$ (the interior of $Z^{n-1}$).
    \end{proof}
      \begin{figure}
         \includegraphics[width=0.62\textwidth]{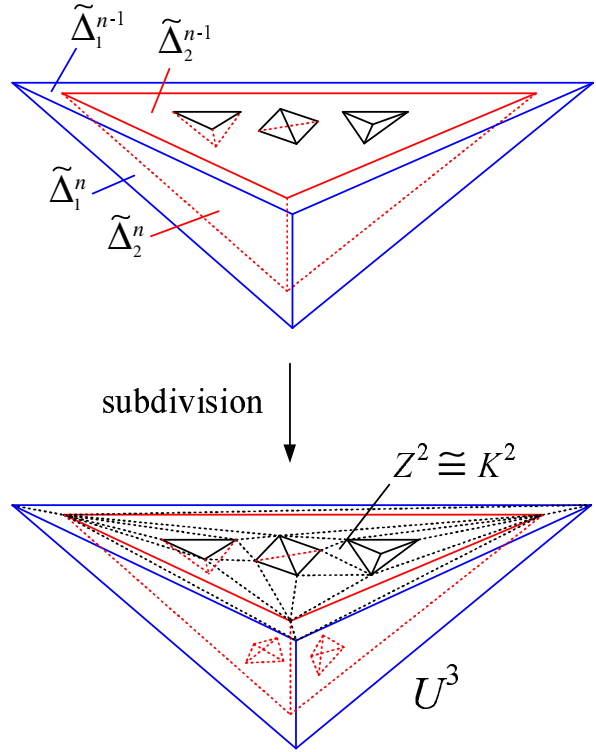}\\
          \caption{A construction of $U^n$ for $n=3$}\label{p:Un}
      \end{figure}

     Such a $U^n$ is called a \textbf{\textit{mold}} \textbf{\textit{$n$-cell}}. Of course,
   mold $n$-cells are not unique. Here,
    we only need to choose a mold $n$-cell $U^n$ for each dimension $n\geq 2$.
    \vskip .2cm

    \begin{definition}
      given an $n$-simplex $\Delta^n$ and an $(n-1)$-face
    $F^{n-1}$ of $\Delta^n$, we first do an $(n-1)$-dimensional bistellar
    $0$-move to $F^{n-1}$ on $\partial \Delta^n$, the boundary
    of $\Delta^n$ becomes a $PL$ $(n-1)$-sphere $L^{n-1}$. Then we
    add a new vertex $v_0$ in the interior of $\Delta^n$ and
    take $v_0 *  L^{n-1}$. The $PL$ $n$-disk $v_0 *  L^{n-1}$
   is called a
   \textit{\textbf{star subdivision of $\Delta^n$ along a face}} $F^{n-1}$ (see
   Figure~\ref{p:Star_Subdivision} for examples). \vskip .2cm
    \end{definition}

    \begin{figure}
         \includegraphics[width=0.92\textwidth]{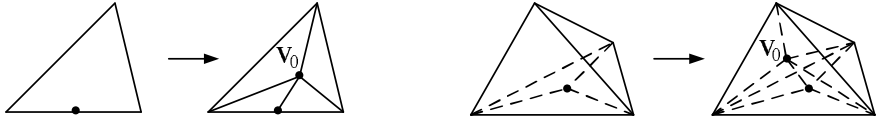}\\
          \caption{Star subdivisions of $\Delta^n$ along a face, $n=2,3$}
          \label{p:Star_Subdivision}
      \end{figure}

    \begin{definition}
       An $n$-simplex $\Delta^n$ in a combinatorial $n$-manifold with boundary
       $(M^n,\xi)$ is
       called \textbf{\textit{exposed}} if there exists some $(n-1)$-face of $\Delta^n$ lying
        on $\partial M^n$.  $\Delta^n$ is called \textbf{\textit{one-face-exposed}}
       if $\Delta^n \cap \partial M^n$ is exactly one $(n-1)$-face of $\Delta^n$.
    \end{definition} \vskip .2cm

    Note that if we do a star subdivision of an exposed $n$-simplex
    $\Delta^n$ in $(M^n,\xi)$ along an $(n-1)$-face of $\Delta^n \cap \partial
    M^n$, we will produce $n$ new one-face-exposed $n$-simplices for $M^n$.
    (see Figure~\ref{p:Star_Subdivision}).\\

    \begin{lem} \label{Lem:Change_f-vector_2}
         For two combinatorial $n$-manifolds with boundary $(M^n_1,\xi_1)$
         and $(M^n_2,\xi_2)$,
         if $\chi(\partial M^n_1) = \chi(\partial M^n_2) $,
         then there exists a
         combinatorial manifold structure $\xi^*_i$ on $M^n_i$ ($i=1,2$)  such that
           $(M^n_i, \xi^*_i) $ is PL-homeomorphic to $(M^n_i, \xi_i)$ and
         $\mathbf{f}(\partial M^n_1, \partial \xi^*_1)
           =\mathbf{f}(\partial M^n_2, \partial \xi^*_2)$.
    \end{lem}
    \begin{proof}
       Since $\chi(\partial M^n_1) = \chi(\partial M^n_2)
      $, by Lemma~\ref{Lem:Virtual_Move}, there exists
      a sequence of $N$ virtual moves in $\Z^{n}_+$ which turn
        $\mathbf{f}(\partial M^n_1,\partial\xi_1)$ to $\mathbf{f}(\partial M^n_2,\partial \xi_2)$.
      But it is not clear whether we can realize these $N$ virtual moves on
      $\mathbf{f}(\partial M^n_1,\partial\xi_1)$ in
      $\Z^{n}_{+}$ by $N$ bistellar moves on $(\partial M^n_1,\partial\xi_1)$
      induced from elementary shellings in $(M^n_1, \xi_1)$.
      Similar to the proof of Lemma~\ref{Lem:Change_f-vector}, we
      can get around this problem by the following steps:
       \begin{enumerate}
         \item By applying the same number of star subdivisions in some
         exposed $n$-simplices
          of $(M^n_i,\xi_i)$ ($i=1,2$) along their $(n-1)$-faces on $\partial M^n_i$,
           we can assume that each $(M^n_i, \xi_i)$
              has more than $N$ one-face-exposed $n$-simplices. Note
              that this process will not change the difference
              $\mathbf{f}(\partial M^n_2,\partial \xi_2)
              -\mathbf{f}(\partial M^n_1,\partial\xi_1)$.
               So the virtual moves in $\Z^{n}_+$ needed to turn
        $\mathbf{f}(\partial M^n_1,\partial\xi_1)$ to $\mathbf{f}(\partial M^n_2,\partial \xi_2)$
        remain the same as before.
            \vskip .2cm

       \item Choose $N$ one-face-exposed $n$-simplices
         $\widetilde{\Delta}^n_1,\cdots, \widetilde{\Delta}^n_N$ in
         $(M^n_1,\xi_1)$ and turn each $\widetilde{\Delta}^n_j$ into a mold
         cell $U^n_j$ such that $U^n_j \cap \partial M^n_1 = Z^{n-1}_j \cong K^{n-1}$.
         At the same time, we do the exactly same operations on
         $N$ one-face-exposed $n$-simplices in
         $(M^n_2,\xi_2)$.
         The new combinatorial manifold structure on $M^n_i$ we get is
         denoted by $\xi'_i$ ($i=1,2$). It is easy to see:
          \begin{enumerate}
           \item $\mathbf{f} (\partial M^n_2,
           \partial \xi'_2) - \mathbf{f} (\partial M^n_1, \partial \xi'_1)
           = \mathbf{f} (\partial M^n_2, \partial \xi_2) -
           \mathbf{f} (\partial M^n_1, \partial\xi_1) $;

            \item on the boundary, each $(\partial M^n_i,\partial
            \xi'_i)$ contains $N$ plump $(n-1)$-cells;
          \end{enumerate} \vskip .2cm

        \item By (2)(a), we need exactly the same number and type of
         virtual moves in $\Z^n_+$ to turn
         $\mathbf{f} (\partial M^n_1, \partial \xi'_1)$ to
         $\mathbf{f} (\partial M^n_2, \partial \xi'_2)$ as we use to turn
           $ \mathbf{f} (\partial M^n_1, \partial\xi_1) $ to
           $\mathbf{f} (\partial M^n_2, \partial \xi_2)$.
           And for each virtual $i$-move used, we can do an elementary
           $(n-1-i)$-shelling in a mold $n$-cell of $(M^n_1, \xi'_1)$ which
           induces a bistellar $i$-move in a plump $(n-1)$-cell on
           $(\partial M^n_1, \partial \xi'_1)$. In the end,
           we get a new combinatorial structure
         $\xi^*_1$ on $M^n_1$ with
           $\mathbf{f}(\partial M^n_1, \partial \xi^*_1)
         = \mathbf{f} (\partial M^n_2, \partial \xi'_2)$.
    \end{enumerate}
       Let $\xi_2^* = \xi_2'$, obviously
       $(M^n_i, \xi^*_i) $ is PL-homeomorphic to $(M^n_i, \xi_i)$, $i=1,2$.
       So we are done.
    \end{proof}
  \vskip .2cm

   \begin{lem} \label{Lem:Change_f-vector_3}
         For two combinatorial $n$-manifolds with boundary $(M^n_1,\xi_1)$
         and $(M^n_2,\xi_2)$,
         if $\chi(M^n_1) = \chi (M^n_2)$ and $\chi(\partial M^n_1) = \chi(\partial M^n_2) $,
          then there exists a
         combinatorial manifold structure $\xi^*_i$ on $M^n_i$ ($i=1,2$)  such that
           $(M^n_i, \xi^*_i) $ is PL-homeomorphic to $(M^n_i, \xi_i)$ and $\mathbf{f}(M^n_1, \xi^*_1)
           =\mathbf{f}(M^n_2, \xi^*_2)$ and $\mathbf{f}(\partial M^n_1, \partial \xi^*_1)
         = \mathbf{f} (\partial M^n_2, \partial \xi^*_2)$.
         \end{lem}
   \begin{proof}
    First of all, since $\chi(\partial M^n_1) = \chi(\partial M^n_2) $,
    by Lemma~\ref{Lem:Change_f-vector_2}, we can
    assume that
     $ \mathbf{f}(\partial M^n_1, \partial \xi_1)
           =\mathbf{f}(\partial M^n_2, \partial \xi_2)$.
    Secondly, $\chi(M^n_1) = \chi (M^n_2)$ and $\chi(\partial M^n_1) =
               \chi(\partial M^n_2) $
   imply that
        \[  f_{-1}(M^n_1, \xi_1)- \frac{1}{2} f_{-1}(\partial M^n_1, \partial\xi_1)
        = f_{-1}(M^n_2, \xi_2)- \frac{1}{2} f_{-1}(\partial M^n_2, \partial\xi_2)
         \]

      \[  \text{Let} \quad \widehat{\mathbf{f}}(M^n_i,\xi_i) :=
             \mathbf{f}(M^n_i, \xi_i)- \frac{1}{2} \mathbf{f}(\partial M^n_i,
             \partial\xi_i), \qquad
             \qquad \]
     \[ \text{and}\quad \widehat{\chi}(M^n_i)=
             f_{-1}(M^n_i, \xi_i)- \frac{1}{2} f_{-1}(\partial M^n_i, \partial\xi_i),
             \quad i=1,2. \]

      Then by Corollary~\ref{Cor:depend_2},
      the parallel statements of Lemma~\ref{Lem:Virtual_Move} and
      Lemma~\ref{Lem:Change_f-vector} still hold if we replace
      \begin{align*}
          \chi(M^n) \ & \longrightarrow \ \widehat{\chi}(M^n_i) \\
          \mathbf{f}(M^n_i,\xi_i) \ & \longrightarrow  \  \widehat{\mathbf{f}}(M^n_i,\xi_i)
       \end{align*}
       Therefore, we can find a combinatorial manifold structure $\xi^*_i$ on
      $M^n_i$ ($i=1,2$) such that
      $(M^n_i, \xi^*_i) $ is PL-homeomorphic to $(M^n_i, \xi_i)$ and
      $\widehat{\mathbf{f}}(M^n_1,\xi^*_1)
       = \widehat{\mathbf{f}}(M^n_2,\xi^*_2) $.
      Notice that the subdivisions and bistellar moves involved
      in our parallel proof for $\widehat{\mathbf{f}}(M^n_i,\xi_i)$ here
       are done in the interior of $M^n_1$ and $M^n_2$.
      So the combinatorial structure $\xi_i$ on $\partial M^n_i$ $(i=1,2)$ is
      not changed, i.e. $(\partial M^n_i, \partial \xi^*_i)= (\partial M^n_i, \partial
      \xi_i)$.
      So we have $ \mathbf{f}(\partial M^n_1, \partial \xi^*_1)
           =\mathbf{f}(\partial M^n_2, \partial \xi^*_2)$, and then
            $\mathbf{f}(M^n_1, \xi^*_1) =\mathbf{f}(M^n_2, \xi^*_2)$.
   \end{proof} \vskip .2cm

   \textbf{Proof of Theorem~\ref{thm:main2}:}
     By the assumptions in the theorem and Lemma~\ref{Lem:Change_f-vector_3} above,
     we can assume  $\mathbf{f}(M^n_1, \xi_1) =\mathbf{f}(M^n_2,\xi_2)$.
     Since the invariant $\Phi$ depends only on the $\mathbf{f}$-vector of a combinatorial manifold,
   then $\Phi(M^n_1,\xi_1) = \Phi(M^n_2,\xi_2)$.
   \hfill $\square$

\end{document}